\documentclass[12pt]{amsart}
\usepackage{amsmath}
\usepackage{amssymb}
\usepackage{amstext}
\usepackage{amscd}
\usepackage[francais, english]{babel}
\usepackage[matrix,arrow,ps]{xy}

\newtheorem{teo}{Theorem}[section]
\newtheorem{prop}[teo]{Proposition}
\newtheorem{lem}[teo]{Lemma}

\newtheorem{conj}[teo]{Conjecture}

\newtheorem{rem}[teo]{Remark}

\newcommand{\CC}{{\mathbb C}}
\newcommand{\RR}{{\mathbb R}}
\newcommand{\ZZ}{{\mathbb Z}}
\newcommand{\QQ}{{\mathbb Q}}

\newcommand{\PP}{{\mathbb P}}

\newcommand{\lto}{\longrightarrow}

\newcommand{\cF}{\mathcal{F}}

\newcommand{\cS}{\mathcal{S}}
\newcommand{\cO}{\mathcal{O}}

\newcommand{\cZ}{{\mathcal Z}}
\newcommand{\cX}{{\mathcal X}}
\newcommand{\cY}{{\mathcal Y}}

\newcommand{\ol}{\overline}

\newcommand{\wt}{\widetilde}

\title{o-minimal flows on abelian varieties.}
\author{Emmanuel Ullmo, Andrei Yafaev}
\address{Ullmo: IHES 35 Route de Chartres, 91440 Bures-sur-Yvette, France}
\email{ullmo@ihes.fr}
 \address{Yafaev: UCL, Department of Mathematics, Gower street, WC1E 6BT, London, UK}
\email{yafaev@math.ucl.ac.uk}

\begin{document}
\maketitle

\begin{abstract}
Let $A$ be an abelian variety over $\CC$ of dimension $n$ and 
$\pi\colon \CC^n \lto A$ be the complex uniformisation. 
Let $X$ be an unbounded subset of $\CC^n$ definable in a suitable 
o-minimal structure. We give a description of the Zariski closure of $\pi(X)$.
\end{abstract}

\section{Introduction.}

Let $A$ be a complex abelian variety of dimension $n$. 
Write $A = \CC^n /\Lambda$ where $\Lambda \subset \CC^n$ is a lattice and let 
$\pi \colon \CC^n \lto A$ be the uniformisation map.

A subvariety $V$ of $A$ is called \emph{weakly special} if 
$V = P + B$ where $P$ is a point of $A$ and $B$ is an abelian subvariety.
The abelian Ax-Lindemann-Weierstrass theorem is the following.

\begin{teo} \label{AxLin}
Let $Y$ be a complex algebraic subset of $\CC^n$.
The components of the Zariski closure of $\pi(Y)$ are 
weakly special subvarieties.
\end{teo}

This theorem is due to Ax (see \cite{Ax1} and \cite{Ax2}) and plays an important 
role in the new proof by Pila and Zannier of the Manin-Mumford conjecture \cite{PZ}.
Note that the paper \cite{PZ} provides a different proof of the abelian Ax-Lindemann-Weierstrass
theorem. 
For a proof close in spirit to the contents of this paper, see Section 9 of \cite{Orr}.
In reality, in this statement, $Y$ can be taken to be only \emph{semialgebraic} ($\CC^n$ being identified with $\RR^{2n}$).

The aim of this paper is to investigate the Zariski closure of the sets $\pi(X)$
where $X$ is definable in an o-minimal structure which is a much wider class of objects.
We refer to the book \cite{VDD} for the notion of a set definable in an o-minimal structure,
in particular the structures $\RR_{an}$ and $\RR_{an,exp}$ (this last structure is actually defined and studied in \cite{MillerVDD}). Just recall that $\RR_{an}$ is 
the o-minimal structure generated by the restricted analytic functions and $\RR_{an.exp}$ is 
additionally generated by the graph of the real exponential.
For a subset $\Sigma$ of $A$, we denote by $Zar(\Sigma)$ its Zariski closure.

To be able to prove anything, we will need to make certain additional assumptions. Firstly, the set $X$ will be assumed to be \emph{unbounded}.
The necessity of this condition can be demonstrated by the following example.
Let $\cF$ be a connected bounded fundamental domain for the action of $\Lambda$ on $\CC^n$.
The restriction of $\pi$ to $\cF$ is definable in $\RR_{an}$.
Let $V$ be any algebraic subvariety of $A$ and let $\wt{V} = \pi^{-1}(V)\cap \cF$. Then
$\wt{V}$ is definable in $\RR_{an}$ and $Zar(\pi(\wt{V})) = V$.

However, when $X$ is an unbounded real analytic manifold, we prove the following.

\begin{teo} \label{t1}
Let $X$ be an unbounded real analytic manifold of $\CC^n = \RR^{2n}$ definable in an o-minimal structure which is an extension of $\RR_{an}$.

Let $V = Zar(\pi(X))$.  For any point $P$ of $\pi(X)$ there is a positive dimensional abelian subvariety $B_P$ of $A$ such that $P + B_P$ is contained in $V$.

In particular, $V$ contains a Zariski dense set of positive dimensional weakly special subvarieties.
\end{teo}

To investigate general definable sets $X$, we
will also impose some mild restrictions on the o-minimal structure.
Let $\cS$ be an o-minimal structure over $\RR$, containing $\RR_{an}$ and whose definable sets
admit an analytic stratification (as defined in \cite{VDD}, Chapter 3).
This condition holds for most `usual' o-minimal structures, for example
$\RR_{an}$ and $\RR_{an,exp}$ (see \cite{MillerVDD}).
We fix such a structure $\cS$ and 
in what follows and by definable, we will mean `definable in $\cS$'.

Next we introduce the notion of \emph{essential Zariski closure}.
Let $X$ be an unbounded definable set as before. 
For $R>0$, let $B(0,R)$ be the open unit ball of centre $0$ and radius $R$.
The variation of the sets $\pi(X \cap B(0,R))$ when $R$ varies is what we call an \emph{o-minimal flow}.
We show that for $R$ large enough, the Zariski closure of the set $\pi(X \backslash (X \cap B(0,R)))$ is constant. 
We call this the \emph{essential} Zariski closure of $\pi(X)$ and denote it by
$Zaress(\pi(X))$.

For an abelian subvariety $B$ of $A$, write $V_B \subset \CC^n$ for the tangent space to $B$ at the origin and $p_B$ for the projection 
$\CC^n \lto V_B$.

We prove the following:

\begin{teo} \label{main_thm1}
Let $X$ be an unbounded definable subset of $\CC^n$. Let $V$ be 
$Zaress(\pi(X))$.

For each point $P$, in $\pi(X)$, there exists a positive dimensional abelian subvariety 
$B_P$ of $A$ such that $P + B_P$ is contained in $V$.

In particular, $V$ 
  contains a Zariski dense set of positive dimensional
weakly special subvarieties.
\end{teo}

We prove a characterisation of subvarieties of an abelian variety containing a Zariski dense set of weakly special subvarieties (see proposition \ref{Zardensity}). 
Let $V$ be such a subvariety. Our proposition \ref{Zardensity} shows that there exist abelian subvarieties $B$ and $B'$ of $A$
such that $A= B + B'$ and $B \cap B'$ is finite, 
$V = B + V'$ where $V'$ is a subvariety of $B'$.

We deduce the following.

\begin{teo}\label{main_thm}
Assume that $X$ is a definable subset of $\CC^n$ such that for all abelian subvarieties $B$ of $A$, 
$p_B(X)$ is unbounded.
Then components of 
$Zaress(\pi(X))$ are weakly special. 
\end{teo}

The strategy of the proof of the theorem \ref{t1} relies on the theory of o-minimality and Pila-Wilkie counting theorem.
Let $X$ be as in the statement and $V$ be the Zariski closure of $\pi(X)$. Using a suitable definable set and applying Pila-Wilkie theorem, we show that there exists a positive dimensional semi-algebraic 
set $W \subset \CC^n = \RR^{2n}$ such that $X + W$ is contained in $\pi^{-1}(V)$.
Applying the Ax-Lindemann-Weierstrass theorem, we then show that 
for any $P$ of $\pi(X)$, there exists a weakly special subvariety $P + B_P \subset V$.

Finally, we would like to point out one possible application of our theorem.

Recall the following theorem of Bloch-Ochiai (see Chapter 9 of \cite{Kobayashi})
which is proved using Nevanlinna theory.

\begin{teo} \label{bloch}
Let $A$ be an abelian variety and $f \colon \CC \lto A$ be a non-constant holomorphic map. Then the Zariski closure of $f(\CC)$ is a translate of an abelian subvariety.
\end{teo}

Theorem \ref{main_thm} imples some cases of theorem \ref{bloch}.

Consider for example $A = \CC^n/\Lambda$ (where $\Lambda$ is a lattice such that $A$ is a simple abelian variety) and $f \colon \CC \lto A$
given by $f(z) = (z, \dots , z , e^z, \dots , e^z) {\rm mod} \Lambda$ with $s$ factors of $z$ and $r$ times of $e^z$ with $r+s = n$. Then consider the set 
$X \subset \CC^n$ given by 
$$
X = \{ (x+iy, \dots , x+iy, e^xe^{iy}, \dots,  e^xe^{iy}) : x \in \RR, y \in [0,2\pi] \}.
$$
Clearly $X$ is unbounded and definable in $\RR_{an,exp}$ and its image in $A$ is contained in  $f(\CC)$. By theorem \ref{main_thm}, the Zariski closure of $f(\CC)$ is $A$ (since $A$ is simple). 

It is not however always possible to ``extract'' such a definable unbounded set 
$X$ from $f(\CC)$ as the example of $(e^z, e^{iz}) \subset \CC^2$ shows.
Indeed, in this example, for any subset $Y \subset \CC$ such that
$f(Y)$ is definable, both the real and imaginary parts of $z \in Y$ must be bounded.

Another (counter)-example is the following. Define the iterated exponential function $exp_n(x)$ by $exp_1 = exp$ and
$exp_{n} = exp \circ exp_{n-1}$.
By Proposition 9.10 of \cite{MillerVDD}, a definable function in $\RR_{an,exp}$ is bounded by $exp_n(x^m)$ for some $n,m$. Therefore a graph of a function which `grows faster' than any $exp_n$ will not satisfy the assumptions of our theorems.
Note that it is a long-standing open problem whether there exists an o-minimal structure 
containing a ``super-exponential'' function.

We conclude this introduction with an open question in the spirit of \cite{UY}. It concerns the topological closure of $\pi(X)$ rather than Zariski closure. 
Recall from \cite{UY} that a \emph{real} weakly special subvariety is defined to be a translate of 
a real subtorus of $A$ (hence not necessarily algebraic).

\begin{conj}
Let $X$ be, as before, an unbounded definable real analytic manifold. We denote by $\ol{\pi(X)}$ the topological closure of $\pi(X)$.

There exists a real analytic submanifold $V$ of $A$ containing a dense subset of real weakly special
subvarieties
such that
$$
\ol{\pi(X)} = \pi(X) \cup V.
$$   
\end{conj}

In section \ref{shuwu}, we prove a characterisation of subvarieties of abelian varieties containing a Zariski dense
subset of weakly special subvarieties, namely that such a subvariety is a union of weakly special ones.
We believe this result and our argument to be of independent interest.

\section*{Acknowledgements.}

The second author is grateful to Alex Wilkie and Gareth Jones 
for useful discussions at the `O-minimality and applications' conference in Konstanz in July 2015.
The second author
is grateful to the IHES for hospitality during his visit in May 2016 when this paper was written in its final form.
The second author gratefully acknowledges 
financial support of the ERC, Project 511343.

We would like to thank the referee for their valuable comments.

\section{Proof of theorem \ref{t1}.}

In this section we assume that $X$ is an unbounded real analytic submanifold of
$\CC^n = \RR^{2n}$ definable in some o-minimal structure which contains $\RR_{an}$. Let $V$ be the Zariski closure of $\pi(X)$ in $A$. 

\subsection{A definable set and point counting.}
The contents of this section are essentially a reproduction of the arguments of Orr from Section 9 of \cite{Orr} with slight adjustments. 

In this section we define 
a certain definable set associated with $X$ and, using Pila-Wilkie theorem, show that this set contains a positive dimensional semi-algebraic subset.

Choose a fundamental set $\cF$ for the action of $\Lambda$ on $\CC^n$
such that $X \cap \cF$ is non-empty. We choose $\cF$ to be an open connected subset of 
$\CC^n$ such that $\ol{\cF}$ is compact and $\Lambda$-translates of $\ol{\cF}$ cover $\CC^n$.
The set $\cF$ is an `open parallelepided'.
 Since $\cF$ is an open subset of $\CC^n$, we have that $\dim(X\cap \cF) = \dim(X)$.
Let $\wt{V}$ be $\cF \cap \pi^{-1}V$. This is a definable set since the o-minimal structure contains $\RR_{an}$ and $\pi$ restricted to $\cF$ is definable in $\RR_{an}$.

Consider the definable set
$$
\Sigma = \{ x \in \CC^n : \dim(X + x)\cap \wt{V} = \dim(X) \}. 
$$


We have the following lemma:
\begin{lem}
If $\lambda \in \Lambda$ and $X \cap (\cF - \lambda) \not= \emptyset$, then 
$\lambda \in \Sigma$.
\end{lem}
\begin{proof}
From $\Lambda$-invariance of $\pi^{-1}V + \lambda = \pi^{-1}V$, 
we see that for $\lambda$ as in the statement (in particular for $\lambda \in \Lambda$),
$X + \lambda \subset \pi^{-1}V$.

It follows that
$$
(X + \lambda)\cap \wt{V} = (X + \lambda)\cap \cF.
$$

As $\cF - \lambda$ is an open subset of $\CC^n$, we see that
$$
\dim( X \cap (\cF - \lambda)) = \dim(X) = \dim ((X + \lambda)\cap \cF)
$$

The conclusion follows.
\end{proof}

Fix a basis $\lambda_1, \dots, \lambda_{2n}$ of $\Lambda$.
Then $\Lambda\otimes \QQ$ is identified with $\QQ^{2n}$.
We define the height of an element 
$\lambda = \sum a_i \lambda_i \in \Lambda$ ($a_i \in \ZZ$) as
$$
H(\lambda) = \max (|a_1|,\dots , |a_{2n}|).
$$
This height thus coincides with the usual height on $\QQ^n$.
 
\begin{prop} \label{counting}
There exists $T_0 \geq 0$ such that for all $T \geq T_0$,
$$
|\{ x \in \Sigma \cap \Lambda : H(x) \leq T\}| \geq T/2.
$$
\end{prop}

\begin{proof}
This is essentially Lemma 9.1 of \cite{Orr}.

The first observation is that if $x_1$ and $x_2$ are two points of $\Lambda$
such that $X \cap (\cF - x_1)$ and $X \cap (\cF - x_2)$
are both non-empty, then $\Sigma \cap \Lambda$ contains at least one point of height $h$
for every $h$ between $H(x_1)$ and $H(x_2)$.

Note that $X$ is path-wise connected in the Euclidean topology. Let $C$ be a path from 
a point in $X \cap (\cF - x_2)$ to a point in $X \cap (\cF - x_2)$.

When $C$ crosses over from $\cF - u_1$ to to an adjacent domain $\cF-u_2$, the heights of $u_1$ and $u_2$ change by at most one.

It follows that for any $h$ between $H(x_1)$ and $H(x_2)$, there is a $u \in \Lambda$ of height $\leq h$
such that $X \cap (\cF - u)$ is not empty.
This $u$ belongs to $\Sigma \cap X$.

By assumption $X$ is unbounded. Thus  as $x$ varies in $\Lambda$ such that $X \cap {\cF - x}$ is non-empty, $h(x)$ goes to infinity.

It follows that there is an $h_0$ such that for any $h > h_0$, $\Sigma \cap \Lambda$ contains at least one point of height $h$.

Take $T_0 = 2h_0$.
\end{proof}

\begin{rem}
The referee has pointed out to us that Tsimerman, in \cite{T}, has made a similar observation. Namely, that in a similar setting an unbounded analytic set should intersect `a lot of fundamental domains'. 
\end{rem}

We now use the following theorem of Pila and Wilkie (\cite{PW}, Theorem 1.8).

For a definable subset $\Theta\subset \RR^n$, we define $\Theta^{alg}$ to be 
the union of all positive dimensional semi-algebraic subsets contained in $\Theta$. We define $\Theta^{tr}$ to be
$\Theta \backslash \Theta^{alg}$.

\begin{teo}[Pila-Wilkie]
Let $\Theta$ be a subset of $\RR^{n}$ definable in an o-minimal structure.
Let $\epsilon > 0$. There exists a constant $c = c(\Theta,\epsilon)$ such that
for any $T \geq 0$,
$$
|\{ x \in \Theta^{tr} \cap \QQ^n : H(x) \leq T \}| \geq c T^{\epsilon}.
$$
\end{teo}

From Proposition \ref{counting} it now follows that $\Sigma^{alg}\cap \Lambda$ is not empty.

Let $W$ be a connected positive dimensional semi-algebraic subset 
contained in $\Sigma$.
For each $w$ in $W$, $\dim( (X + w) \cap \wt{V}) = \dim(X)$ 
and hence an analytic component of $(X + w) \cap \cF$ is contained in $\pi^{-1} V$. By analytic continuation, we see that $X + w \subset \pi^{-1}V$.
We have proved:

\begin{prop}\label{semialg}
With the notations and assumptions of this section, there exists a positive dimensional semialgebraic subset $W$ such that
$$
X + W \subset  \pi^{-1}V.
$$
\end{prop}

\subsection{Final argument.}

We use the following lemma whose proof can for example be found in \cite{Orr}, Lemma 8.1.
\begin{lem}
Let $\cZ$ be a connected complex analytic subset of $\CC^g$.
Let $\cX$ be a connected irreducible semialgebraic set contained in $\cZ$. 
Then there is a complex algebraic variety $\cY$ such that
$\cX \subset \cY \subset \cZ$.
\end{lem}

By proposition \ref{semialg} and the above lemma, we see that for any $x \in X$, there exists
a positive dimensional complex algebraic subset $Y_x$ containing $X$ and contained in $\pi^{-1}(V)$.
By the abelian Ax-Lindemann-Weierstrass theorem \ref{AxLin}, the Zariski closure of 
$\pi(Y_x)$ is a union of weakly special subvarieties of $V$. Therefore, $V$ contains a subvariety of the form $P + B_P$ where $P=\pi(x)$ and 
$B_P$ is a positive dimensional abelian subvariety of $A$. This finishes the proof of theorem \ref{t1}.

\section{Cell decomposition and essential closure.}

In this section we consider an unbounded definable set $X \subset \CC^n$. We refer to section 8 of \cite{MillerVDD} for the definition of a real analytic cell. What is relevant to us is that a real analytic cell in $\RR^n$  is a definable real analytic submanifold, definable-analytically isomorphic to
$\RR^m$ for some $m \leq n$.
By Theorem 8.9 of \cite{MillerVDD}, there is a finite number of analytic cells $X_1, \dots , X_k$
such that $X$ is a disjoint union of the $X_k$.

\begin{prop}
The essential closure $Zaress(\pi(X))$ is the union of $Zar(\pi(X_i))$ where $X_i$s are the unbounded cells.

\end{prop}
\begin{proof}
We start with a lemma.
\begin{lem}\label{continuation}
Let $Z$ be  a real analytic manifold in $\CC^n$ and $U \subset Z$ an open subset.

Then
$$
Zar(\pi(U)) = Zar(\pi(Z))
$$ 

In particular, if $Z$ is an analytic unbounded submanifold of $\CC^n$, then
$$
Zaress(\pi(Z)) = Zar(\pi(Z))
$$
\end{lem}
\begin{proof}
One inclusion is obvious. 

Write $Zar(\pi(U)) \subset \PP^m$ for some $m$ and let $s \in H^0(\PP^m, \cO(l))$ for 
$l\geq 1$ such that $s$ is zero on $\pi(U)$.
Then $s \circ \pi$ is zero on $U$ and by analytic continuation $s \circ \pi$ is zero on $Z$.
It follows that $s$ is zero on $\pi(Z)$, hence $Zar(\pi(Z)) \subset Zar(\pi(U))$.
\end{proof}

Let $X = X_1 \coprod \dots \coprod X_k$ be a cell
decomposition of $X$. 
For $R$ large enough, $X \cap B(0,R)$ contains the union of all the bounded cells in the above decomposition.

We have 
$$
Zaress(\pi(X)) = \bigcup_{\{i : X_i \text{unbounded} \}} Zaress(\pi(X_i)).
$$

By Lemma \ref{continuation}, for an unbounded cell  $X_i$,
$$
Zaress(\pi(X_i)) = Zar(\pi(X_i)).
$$

The result follows.





\end{proof}

\section{Characterisation of subvarieties containing a dense set of weakly special subvarieties.}\label{shuwu}
In this section we prove a proposition which we believe to be of 
independent interest.

Let $A$ be an abelian variety and $V$ a subvariety of $A$.
Define the stabiliser of $V$ as
$$
{\rm Stab}(V) = \{ P \in A : P + V = V\}.
$$

Recall that for an abelian subvariety $B$ of $A$, there exists an abelian subvariety $B'$ such that $A = B + B'$
and $B \cap B'$ is finite. We always refer to $B$ and $B'$ as above.

\begin{prop} \label{Zardensity}
Let $V$ be an irreducible subvariety of $A$.

\begin{enumerate}
\item
Assume $\dim {\rm Stab}(V) > 0$.

Then there exists abelian subvarieties $B$ and $B'$ of $A$ such that
$A = B + B'$ and $V = B + V'$ where $V'$ is a subvariety of $B'$. 
\item
Assume that ${\rm Stab}(V)$ is finite. Then the set of positive dimensional weakly special
subvarieties contained in $V$ is not Zariski dense.

\item
Assume again that ${\rm Stab}(V)$ is finite. 
Let $\Sigma$ be the set of all positive dimensional weakly special subvarieties contained in $V$.

For an abelian subvariety $B \subset A$, denote by $B'$ an abelian subvariety such that
$A = B + B'$.

There exists a finite set $B_1, \dots , B_r$ of abelian subvarieties of $A$ and $W_1, \dots, W_r$
of subvarieties of $B'_i$ such that
$$
Zar(\Sigma) = \bigcup_{i=1}^r B_i + W_i.
$$
\end{enumerate}
\end{prop}
\begin{proof}
Assume $\dim {\rm Stab}(V) > 0$ and let $B$ be the neutral component of ${\rm Stab}(V)$.

Let $B'$ be an abelian subvariety such that $A = B + B'$ and let $\psi \colon A \lto A/B$ be the quotient.
Let $V'$ be $\psi|_{B'}^{-1}(\psi(V))$. Then 
$$
V = \{ B + x : x \in V \} = \{B + x : x \in V'\} = B + V'.
$$
This proves (1).


We will now prove (2).
Assume that ${\rm Stab}(V)$ is finite.
We start by reducing to the case where ${\rm Stab}(V) = \{ 0 \}$.
Let $A' = A/{\rm Stab}(V)$ and let $\phi \colon A \lto A'$ be the 
quotient map and let $V' = \phi(V)$. Note that $\phi^{-1}(V') = V + {\rm Stab}(V) = V$.
We claim that ${\rm Stab}(V') = \{ 0 \}$. 
Let $P \in {\rm Stab}(V')$ and $Q \in \phi^{-1}(P)$.
We have 
$$
\phi(Q + V) = P + V' = V'
$$
It follows that $Q + V \subset \phi^{-1}(V') = V$ and for dimension reasons $Q + V = V$.
Hence $Q \in {\rm Stab}(V)$ and $P = \phi(Q) = 0$.

As the conclusion of (2) holds for $V$ if and only if it holds for $V'$, we may therefore assume that ${\rm Stab}(V) = \{ 0 \}$.

For $m > 1$, consider the map
$$
\phi_m \colon V^m \lto A^{m-1}
$$
defined by
$$
\phi_m(x_1, \dots, x_m) = (x_1 - x_2, \dots , x_m - x_{m-1}).
$$
By \cite{Zhang}, Lemma 3.1, there exists $m > 1$ such that the map $\phi_m$ is
a generic embedding.

Let $P + B$ be a positive dimensional weakly special subvariety contained in $V$. Then $\phi_m((P+B)^m) = B^{m-1}$. 
The map $\phi_m$ is therefore not injective on $(P+B)^m$. 
Therefore $V$ can not contain a Zariski dense set of positive dimensional subvarieties of the form $P+B$.
This proves (2).

Let us now prove (3).
Let $\Sigma$ as in the statement, the set of all positive dimensional weakly special subvarieties contained in $V$
and let $W$ be a component of $Zar(\Sigma)$.
Then $W$ contains a Zariski dense set of weakly special subvarieties and by (2), ${\rm Stab}(W)$ is
positive dimensional. It follows from (1) that $W = B + W'$ where $B$ is an abelian subvariety of $A$ and $W'$ a subvariety of $B'$. Since $Zar(\Sigma)$ has finitely many components, the conclusion of (3) follows. 
\end{proof}

\begin{rem}
The geometric aspect of 
Lang's conjecture predicts that given a variety of general type $V$, the union of subvarieties,
not of general type, is not Zariski dense.
It is a known fact that a subvariety $V$ of an abelian variety is of general type if and only if
${\rm Stab}(V)$ is finite.
Therefore, our proposition \ref{Zardensity} implies the geometric Lang's conjecture for subvarieties of abelian varieties.
\end{rem}

\begin{rem}
This proposition is an abelian analogue of the result of the first author (see \cite{Ul}) in the hyperbolic case which is proved by 
completely different methods.
\end{rem}

\section{Proof theorems \ref{main_thm1} and \ref{main_thm}.}

In this section we deduce theorems \ref{main_thm1} and \ref{main_thm} from the preceding results.

Let $A$ and $X$ be as in the assumptions of Theorem \ref{main_thm1}.
Let $V$  be a component of the essential Zariski closure of $\pi(X)$.

In section 3 we have seen that $Zaress(\pi(X))$ is a finite union of Zariski closures of sets of the form 
$\pi(Y)$ where $Y$ is an unbounded definable real analytic submanifold of $\CC^n$.
Therefore, the conclusion of theorem \ref{main_thm1} follows from theorem \ref{t1}.

Let now $X$ be as in \ref{main_thm}. By theorem \ref{main_thm1}, $V=Zaress(X)$ contains a Zariski dense set of positive dimensional weakly special subvarieties. From proposition \ref{Zardensity}, we deduce that $V$ is of the form $V = B + V'$ where $B$ is a positive dimensional abelian subvariety 
 of $A$ and $V'$ is a subvariety of $B'$.
 Reiterating the argument with $B'$ and $V'$, we conclude that components of $V$ are weakly special.


\begin{thebibliography}{99}


\bibitem{Ax1} J. Ax, {\it On Schanuel's conjecture}, Annals of
  Math. {\bf 93} (1971), 1-24.
  
  \bibitem{Ax2} J. Ax, {\it Some topics in differential algebraic geometry I: Analytic subgroups of algebraic groups,}
  Amer. J. Math. {\bf 94} (1972), 1195-1204. 


\bibitem{Kobayashi} S. Kobayashi, {\it Hyperbolic Complex Spaces.}
Vol. 318. A Series of Comprehensive studies in Mathematics, Springer 1998.

\bibitem{MillerVDD} C. Miller, L. Van Den Dries, {\it On the real exponential field with restricted analytic functions.} Israel Journal of Math, Vol 85, 1994, 19-56.


\bibitem{Orr} M. Orr, {\it Introduction to abelian varieties and the 
Ax-Lindemann-Weierstrass theorem.} In ``O-Minimality and Diophantine Geometry''. LMS Lecture Notes Series, 421, 2015.

\bibitem{PW} J. Pila, A. Wilkie, {\it The rational points of a definable set.} Duke Math. Journal, 133(3), 591-616, 2006.

\bibitem{PZ} J. Pila, U. Zannier, {\it Rational points in periodic analytic sets and the Manin-Mumford conjecture, }
Rend. Math. Acc. Lincei {\bf 19} (2008), 149--162.

\bibitem{RU} N. Ratazzi, E. Ullmo, {\it Galois+Equidistribution=Manin-Mumford.}  Summer School Arithmetic geometry, 419-430, Clay Math. Proc., 8, Amer. Math. Soc., Providence, RI, 2009.

\bibitem{Ul} E. Ullmo, {\it Applications du theor\`eme d'Ax-Lindemann hyperbolique.}
Compositio Mathematica, Vol 150, Issue 02,175- 190, 2014.

\bibitem{UY} E. Ullmo, A. Yafaev, {\it Algebraic flows on abelian varieties.} Crelle's Journal. To appear.

\bibitem{T} J. Tsimerman, {\it Ax-Schanuel and o-minimality.} In ``O-Minimality and Diophantine Geometry''. LMS Lecture Notes Series, 421, 2015.

\bibitem{VDD} L. Van Den Dries, {\it Tame topology and o-minimal structures.}
LMS Lecture Notes Series, {\bf 248}, 1998.

\bibitem{Zhang} S-W. Zhang, {\it Equidistribution of small points an abelian varieties.} Annals of Maths, {\bf 147}, 159-165, 1998.

\end{thebibliography}
\end{document}